\documentclass[11pt]{amsart}
 
\usepackage{hyperref, color}
\newtheorem{theorem}{Theorem}[section]
\newtheorem{lemma}[theorem]{Lemma}
\newtheorem{proposition}[theorem]{Proposition}
\newtheorem{corollary}[theorem]{Corollary}

\theoremstyle{definition}

\theoremstyle{remark}
\newtheorem{remark}[theorem]{Remark}

\numberwithin{equation}{section}

\begin{document}

\title
[The Hawking mass for CMC surfaces ]
{On the Hawking mass for CMC surfaces in positive curved 3-manifolds}



\author[Luiz Melo]{Luiz Ricardo Abreu Melo}

\address{Instituto de Matem\'atica - Universidade Federal de Alagoas, Macei\'o - Brazil}
\email{luiz.melo@im.ufal.br}


\date{\today}


\commby{}

\begin{abstract} 
In this paper, we present two rigidity results for stable constant mean curvature (CMC) surfaces immersed in $3$-manifolds with positive scalar curvature, assuming that the Hawking mass is zero. In the first result, we assume the surface to be approximately round, while in the second result, we consider surfaces invariant under an even symmetry.
\end{abstract}

\maketitle


\section{Introduction}
In the context of general relativity, the concept of mass arises from the need to establish a connection between general relativity and classical Newtonian gravity. One of the earliest attempts in this direction is the ADM mass 
for asymptotically flat manifolds. In the time-symmetric case, it relates the local density of matter (equivalent to energy) on a Riemannian manifold $(M, g)$ to the scalar curvature, denoted by $S_g$. One of the most significant applications of the ADM mass is the Positive Mass Theorem (PMT), as proven by Schoen and Yau \cite{sy79} and Witten \cite{Witten}. This theorem states that the ADM mass of a complete asymptotically flat three-manifold with nonnegative scalar curvature is always nonnegative and equals zero only if $M$ is isometric to flat Euclidean space.

Another important result in the context of general relativity is the Penrose inequality, which states that if $(M,g)$ is a complete asymptotically flat three-manifold with non-negative scalar curvature and boundary $\Sigma = \partial M \ne 0$, then the ADM mass $M$ is greater than or equal to the Hawking mass of $\Sigma$, that is, $\sqrt{|\Sigma|/16\pi}$, where $\Sigma$ is an outermost minimal surface, with equality if, and only if, $M$ is isometric to half of the Schwarzschild metric $\mathbb{R}^3 \setminus \{0\}.$

It is natural to consider $\sqrt{|\Sigma|/16\pi}$ as the quasi-local mass of a black hole (minimal surface) $\Sigma$. In a more general sense, in correspondence with Newtonian gravity, it is expected that any bounded region $(\Omega, g)$ possesses a quasi-local mass that takes into account both the density of matter (measured  in this case by the scalar curvature $S_g \geq 0$) and some influence from the gravitational field. The positivity of the total mass established by the PMT motivates certain properties that we would expect from a mass definition. Among the proposed candidates for quasi-local mass, one of the most significant is the Hawking mass \cite{hawking}.

In \cite{bart2003}, Bartnik proposed the problem of investigating the rigidity of the Hawking mass.
That is, what can be inferred about the ambient manifold when the Hawking mass vanishes for a certain surface, and what is the role of 2-spheres (black holes) in this context. 
In general, the t Hawking mass  of  a embedded surface $\Sigma\subset M^3$  is defined as 
\begin{equation}\label{hawkingmass}
m_\text{\rm H}(\Sigma)
=\bigg(\dfrac{|\Sigma|}{16\pi}\bigg)^{1/2}
\bigg( 1-\dfrac{1}{16\pi}\int_\Sigma \mathrm{H}^2\,da-\dfrac{\Lambda}{24\pi}\,|\Sigma|\bigg),
\end{equation} 
where $\Lambda=\inf_M S_g$, $|\Sigma|$ denotes the area of $\Sigma$, and $H$ is the mean curvature function of $\Sigma$ (see \cite{max}).
In particular, assuming $\inf_M S_g = 6$ we get
 \begin{eqnarray}\label{hm2}
     m_\text{\rm H}(\Sigma)
=\dfrac{|\Sigma|^{1/2}}{(16\pi)^{3/2}} \left( 16\pi-\int_\Sigma (\mathrm{H}^2+4)\,da \right).
 \end{eqnarray}

In \cite{sun}, Sun solved this  rigidity problem for nearly round surfaces. By studying the eigenvalues of the Jacobi operator for stable CMC surfaces with zero Hawking mass, he transferred the rigidity problem to the trivial solution of the mean field equation. 
In fact, Sun proved that if the solution is sufficiently close to the trivial solution, which is zero, then this solution is unique, and the surface must be isometric to $\mathbb{S}^2 \subset \mathbb{R}^3$, and then he was able to invoke some geometric rigidity theorems to conclude  that these regions are respectively balls in $\mathbb{R}^3$ and $\mathbb{H}^3$, under curvature assumptions of $S_g\geq 0$ and $S_g \geq -6$. This raises the question of whether for $S_g \geq 6$, we would have a similar result for $\mathbb{S}^3$.

Later, in \cite{shi2019}, Shi et al. studying the solutions to the mean field equation established a new rigidity result for stable CMC spheres  with even symmetry, still assuming either
$S_g\geq 0$ or $S_g\geq -6$.

In this paper, we address the rigidity problem concerning stable constant mean curvature (CMC) spheres within the context of positive curvature. To establish the rigidity of the ambient manifold, it is necessary for us to assume positive Ricci curvature, allowing us to apply the Hang-Wang theorem as outlined in \cite{hangwang}.
We recall that a Riemannian surface $\Sigma$ is \emph{approximately a round sphere} if the Gaussian curvature $K_\Sigma$ is in $C^0$ and 
    \begin{eqnarray*}
     \left|\dfrac{|\Sigma|}{4\pi }K_\Sigma - 1\right|_{C^0}< \epsilon_0   
    \end{eqnarray*}
for some universal constant $\epsilon_0 \ll 1$, where $|\Sigma|$ denotes the area of $\Sigma$.

Using this notation, and denoting by $\mathbb S^2(r)$ the round sphere of radius $r$  in $\mathbb R^3$, we have the following extension of Sun's theorem:

\begin{theorem}\label{teo1}
    Let $(M^3, g)$ be a complete Riemannian manifold with $S_g\geq 6$, and  $\Sigma$ be a stable sphere with constant mean curvature  in $M^3$. If $m_H(\Sigma) = 0$ and $\Sigma$ is approximately a round sphere, then  $\Sigma$ is isometric to $\mathbb{S}^2(\sqrt{|\Sigma|}/{4\pi})$. Moreover, if we assume $\textrm{Ric}_g \geq 2$, then the mean convex region $\Omega\subset M$ whose boundary is $\Sigma$ is isometric to a geodesic ball in $\mathbb{S}^3_+$ whose boundary is $ \mathbb{S}^2(\sqrt{|\Sigma|}/{4\pi})$.
\end{theorem}

In our next result, we extend the ideas presented in \cite[Theorem 2]{shi2019}  to the positive curvature case. 
We recall that $\Sigma$ has 
the even symmetry indicating the presence of an isometry $\rho: \Sigma \rightarrow \Sigma$ satisfying the conditions $\rho^2 = \text{id}$ and $\rho(x) \neq x$ for $x \in \Sigma$.

\begin{theorem}\label{teo2}
Let $(M^3, g)$ be a complete Riemannian manifold with $S_g\geq 6$, and  $\Sigma$ be a stable sphere with constant mean curvature  in $M^3$. If $m_H(\Sigma) = 0$ and $\Sigma$ has the even symmetry, then  $\Sigma$ is isometric to $\mathbb{S}^2(\sqrt{|\Sigma|}/{4\pi})$. Moreover, if we assume $\textrm{Ric}_g \geq 2$, then the mean convex region $\Omega\subset M$ whose boundary is $\Sigma$ is isometric to a geodesic ball in $\mathbb{S}^3_+$ whose boundary is $ \mathbb{S}^2(\sqrt{|\Sigma|}/{4\pi})$.
\end{theorem}

In the case where $M$ is compact and $\Sigma$ is minimal, we obtain stronger rigidity results.

\begin{corollary}
Let $(M^3, g)$ be a compact Riemannian manifold with $S_g\geq 6$, and  $\Sigma$ be a stable minimal sphere   in $M^3$ for volume preserving variations such that $m_H(\Sigma) = 0$. If $\Sigma$ is approximately a round sphere or if $\Sigma$ has the even symmetry, then  $\Sigma$ is isometric to $\mathbb{S}^2$. Moreover, if we assume $\textrm{Ric}_g \geq 2$, then $M^3$ is isometric to $\mathbb{S}^3.$
\end{corollary}

In fact, in this scenario, $\Omega\subset M$ has exactly two connected components of $M\setminus\Sigma$ (see \cite{Lawson70}), and we can conclude from Hang-Wang \cite{hangwang} that each component is isometric to a hemisphere.

\subsection*{Acknowledgments}
This work is part of the author’s Ph.D. thesis at Federal University of Alagoas, Brazil. The author would like to thank his advisor Marcos Cavalcante for his constant encouragement and advice. The author was partially supported by CAPES-Brazil.

\section{Preliminaries results}

We recall that a CMC surface is stable if the second variation of the area functional is nonnegative for volume preserving  variations.  It is well known that  the second variation of the area functional is given by a quadratic  form associated to the Jacobi operator $J$ acting on functions of zero mean and defined as 
\begin{eqnarray*}
    J  = \Delta_\Sigma + |A|^2 + \text{Ric}(N, N),
\end{eqnarray*}
where $\Delta_\Sigma$ is the Laplace-Beltrami operator and $A$ is the shape operator of $\Sigma$. We recall that the index of $\Sigma$ is the number of negative eigenvalues of $J$, counted with multiplicity, and it indicates the number of linearly independent volume preserving variations which decreases area. $\Sigma$ is said to be stable when the index is zero.

It was shown by Christodoulou and Yau \cite{c1988} that the Hawking mass (\ref{hawkingmass}) for $\Lambda=0$ is non-negative for stable spheres in Riemannian manifolds with scalar curvature $S_g \geq 0$. A similar result for manifolds with scalar curvature $S_g \geq -6$ was proven by Chodosh in \cite{cho2016}. 
Simirlarly to the results in \cite{c1988, cho2016}, we use a Hersch-type trick to prove the nonnegativity of the Hawking mass (\ref{hm2}) when $\Lambda=6$.  
\begin{lemma}\label{6sun}
Let $(M^3, g)$ be a complete Riemannian manifold with scalar curvature $S_g \geq 6$. If $\Sigma$ is a stable CMC sphere, then $m_H(\Sigma) \geq 0$.
\end{lemma}
\begin{proof}
Since $\Sigma$ is stable, we have
\begin{eqnarray}\label{indice1}
\int_M (|\nabla f|^2 - 
\left(|A|^2 + \text{Ric}(N, N) \right)f^2)da \geq 0 ,
\end{eqnarray}  
for all $f \in C^{\infty}(\Sigma)$ with $\int_\Sigma f \, da = 0$.
By \cite[Lemma 5.1]{CFP15}, there exists a conformal mapping $\varphi : \Sigma \rightarrow \mathbb{S}^2$ such that $\text{deg}(\varphi)\leq [\frac{g+3}{2}]$ and
$\int_\Sigma \varphi_i  da =0$,
where $\varphi = (\varphi_1, \varphi_2, \varphi_3)$ by considering the canonical embedding $\mathbb S^2\subset \mathbb R^3$.
Using $\varphi_i$ as test functions in (\ref{indice1}), we have
\begin{eqnarray*}
\int_{\Sigma} |\nabla \varphi_i|^2da  \geq \int_{\Sigma} (|A|^2 + \text{Ric}(N, N)) \varphi_i^2 da.
\end{eqnarray*}

Using the fact that $\varphi$ is conformal and 
$\Sigma$ has genus zero, we have
\begin{eqnarray*}
\sum_i \int_{\Sigma} |\nabla \varphi_i|^2 da = \int_\Sigma |\nabla \varphi|^2 da= 2 |\mathbb {S}^2|\text{deg}(\varphi) \leq 8\pi \Big[\frac{g+3}{2} \Big] = 8\pi,
\end{eqnarray*}
and using that $|\varphi|^2 = \sum \varphi_i^2 = 1$, we obtain
\begin{eqnarray*}
 \int_{\Sigma} (|A|^2 + \text{Ric}(N, N)) da \leq 8\pi. 
\end{eqnarray*}

By the Gauss equation
\begin{eqnarray*}
|A|^2 + \text{Ric}_g(N,N) &=& \frac{1}{2} S_g - K_\Sigma + \frac{1}{2} (\mathrm{H}^2 + |A|^2) \\
&=& \frac{1}{2} (S_g + |A^0|^2)  + \frac{3}{4} \mathrm{H}^2  - K_\Sigma,
\end{eqnarray*}
where we used $|A|^2 = |A^0|^2 + \frac{1}{2} \mathrm{H}^2$. 
Thus we obtain
\begin{eqnarray*}
 \frac{1}{2} \int_{\Sigma} (S_g + |A^0|^2 )da + \frac{3}{4} \int_{\Sigma} \mathrm{H}^2 da  - \int_{\Sigma} K_\Sigma da \leq 8\pi. 
\end{eqnarray*}

Using our assumption on the scalar curvature, we have
\begin{eqnarray*}\label{h2}
16\pi - \int_{\Sigma}( \mathrm{H}^2 + 4) da  \geq \frac{2}{3} \int_{\Sigma} |A^0|^2  da\geq 0,
\end{eqnarray*}
and it implies $m_H(\Sigma)\geq 0$ as claimed. 
\end{proof}

    We will now need the following lemma due to El Soufi and Ilias \cite{el1992}, which provides an optimal estimate of the second eigenvalue of a Schrödinger operator. The lemma also establishes rigidity for the second eigenvalue, which can then be applied to the Jacobi operator of CMC spheres.
\begin{lemma}\label{lemma3sun}\cite{el1992} For any continuous function $q$ on a surface $\Sigma$, we have
\begin{eqnarray*}
\lambda_1( \Delta_\Sigma - q ) |\Sigma| \leq 2 A_c(\Sigma) + \int_\Sigma q da.
\end{eqnarray*}
Equality holds if and only if $\Sigma$ admits a conformal map to $\mathbb{S}^2$ whose components are the first eigenfunctions. 
If $\Sigma$ has genus zero, then equality implies that $\Sigma$ is conformal to $\mathbb{S}^2$ in $\mathbb{R}^3$, and $q$ is given by the energy density of a Möbius transformation. 
Here, $\lambda_1$ is the first eigenvalue of $\Delta_\Sigma - q$, and $A_c(\Sigma)$ is the conformal volume, which is $4\pi$ for a sphere.
\end{lemma}

Using the lemma above, we obtain a characterization for stable CMC spheres with zero Hawking mass, generalizing  \cite{sun}.

\begin{proposition}\label{prop1sun}
Let $(M^3, g)$ be a complete Riemannian manifold with scalar curvature $S_g \geq 6$. If $\Sigma$ is a stable CMC sphere with $m_H(\Sigma) = 0$, then the fist eigenvalue $\lambda_1( \Delta - K_\Sigma ) = \frac{12\pi}{|\Sigma|}$ with three eigenfunctions $\varphi_1, \varphi_2, \varphi_3$, where $\int_\Sigma \varphi_i  da =0$ and $\sum \varphi_i^2 = 1$. In particular, $\Sigma$ is totally umbilical, is conformal to $\mathbb{S}^2$ in $\mathbb{R}^3$ and $\sum |\nabla \varphi_i|^2 = \frac{12\pi}{|\Sigma|} - K_\Sigma$, which is independent of the eigenfunctions.
\end{proposition}

\begin{proof}
If $m_H(\Sigma) = 0$ then all the inequalities in Lemma \ref{6sun} become equalities, we have $\int_{\Sigma} (\mathrm{H}^2 + 4) da = 16\pi$, $S_g|_\Sigma = 6$ and $|A^0|^2 = 0$ on $\Sigma$. It follows that $\mathrm{H}^2 = \frac{16\pi}{|\Sigma|} - 4$, and the Jacobi operator becomes
\begin{eqnarray}
J=\Delta_\Sigma+3
+\frac{3}{4}
\bigg(\frac{16\pi}{|\Sigma|} - 4\bigg)- K_\Sigma
=  \Delta_\Sigma  - K_\Sigma + \frac{12\pi}{|\Sigma|}. 
\end{eqnarray}

Using Lemma \ref{lemma3sun}, we have
\begin{eqnarray*}
0\leq  \lambda_1(J)|\Sigma|  \leq 8\pi + \int_\Sigma \left( K_\Sigma - \frac{12\pi}{|\Sigma|}\right) da = 8\pi+4\pi -12\pi=0,
\end{eqnarray*}
so all inequalities are equalities, and
in particular we have
\begin{eqnarray*}
\lambda_1 (\Delta_\Sigma - K_\Sigma) = \frac{12\pi}{|\Sigma|},
\end{eqnarray*}
with three eigenfunctions $\varphi_1, \varphi_2, \varphi_3$ satisfying $\int_\Sigma \varphi_i \psi da = 0$ and $\sum \varphi_i^2 = 1$. 

Therefore, 
\begin{eqnarray}\label{3.12sun}
\Delta_\Sigma \varphi  - K_\Sigma \varphi + \frac{12\pi}{|\Sigma|} \varphi = 0,
\end{eqnarray}
where $\Delta_\Sigma \varphi = (\Delta_\Sigma \varphi_1, \Delta_\Sigma \varphi_2, \Delta_\Sigma \varphi_3)$. 
From $|\varphi|^2 = \sum {\varphi_i}^2 = 1$, we have
\begin{eqnarray*}
0 = \Delta_\Sigma |\varphi|^2 = 2\sum\varphi_i\Delta_\Sigma \varphi_i+ 2\sum  |\nabla  \varphi_i|^2.
\end{eqnarray*}

Using this identity in \eqref{3.12sun}, we obtain
\begin{eqnarray}\label{3.14sun}
\sum|\nabla \varphi_i|^2 = \frac{12\pi}{|\Sigma|} - K_\Sigma.
\end{eqnarray}
\end{proof}

In the following, let us assume that $\Sigma$ is a 
stable CMC surface in $M^3$ with zero Hawking mass, and let $\varphi : \Sigma \rightarrow \mathbb{S}^2 \subset \mathbb{R}^3$ be the conformal map given in Proposition \ref{prop1sun}. Denote the metric on $\Sigma$ as $g = \varphi^* \left(e^{u}g_0\right)$, where $g_0$ is the canonical metric on $\mathbb{S}^2$. According to the definition of a conformal map,
\[
e^{-u} = \frac{1}{2} |\nabla \varphi|^2=\frac{1}{2}\sum |\nabla \varphi_i|^2,
\]
and thus using identity (\ref{3.12sun}) we get
\begin{eqnarray}\label{3.15sun}
e^{-u}=\frac{6\pi}{|\Sigma|} - \frac{1}{2}K_\Sigma.
\end{eqnarray}


The Gauss formula for the conformal change of metric yields
\begin{eqnarray}\label{3.16sun}
K_{\Sigma}  = e^{-u} \left(1 - \frac{1}{2} \Delta_{g_0} u\right).
\end{eqnarray}

Therefore, substituting \eqref{3.16sun} into \eqref{3.15sun}, we obtain

\begin{eqnarray}\label{3.17'sun}
\Delta_{g_0} u = 6 - \frac{24\pi}{|\Sigma|}e^u.
\end{eqnarray}

Now we consider a change of variable, setting $u = v + w$ where $v$ is the constant satisfying $e^v = \frac{|\Sigma|}{4\pi}$. So $\Delta_{g_0} u =\Delta_{g_0} w$ and $e^u=\frac{|\Sigma|}{4\pi}e^w$.
This leads us to
\begin{eqnarray}\label{3.17sun}
\Delta_{g_0} w = 6 - 6e^w,
\end{eqnarray}
and in particular we get
\begin{eqnarray}\label{3.18}
\frac{|\Sigma|}{4\pi}K_\Sigma - 1 = e^{-u} \bigg(3e^u - 2 \frac{|\Sigma|}{4\pi}\bigg) - 1 = 2(1 - e^{-w}).
\end{eqnarray}

 If we can show that \eqref{3.17sun} admits only the trivial solution, we will conclude that $\Sigma$ is isometric to $\mathbb{S}^2 (\sqrt{|\Sigma|}/{4\pi})$. Indeed, it follows that $g = \varphi^*(e^ug_0) = \varphi^*(\frac{|\Sigma|}{4\pi} g_0)$.

\section{Proofs of the theorems}

\subsection{Proof of Theorem \ref{teo1}}
According to Proposition \ref{prop1sun},
if $m_H(\Sigma) = 0$  then, $\mathrm{H}^2 =\frac{16\pi}{|\Sigma|} - 4, \Sigma$ is totally umbilical, $\mathrm{A} = \frac{\mathrm{H}}{2}\textrm{Id} = \sqrt{\frac{4\pi}{|\Sigma|} -1} \textrm{Id}$, which coincides with the second fundamental form of $\mathbb{S}^2(\sqrt{|\Sigma|}/{4\pi})$ immersed in $\mathbb{S}^3$, where $\textrm{Id}$ denotes the identity operator and the Jacobi operator becomes
\begin{eqnarray*}
J=  \Delta_\Sigma  - K_\Sigma + \frac{12\pi}{|\Sigma|}.
\end{eqnarray*}

From equation (\ref{3.18}), 
if $\Sigma$ is approximately a sphere, then $\frac{|\Sigma|}{4\pi}K_\Sigma$ is $C^0$ close to 1, implying that $w$ is $C^0$ close to 0.
By \cite[Lemma 10]{sun},
we conclude that $w=0$,  thus
$\Sigma$ is isometric to $\mathbb{S}^2(\sqrt{|\Sigma|}/{4\pi})$,  the sphere of constant Gaussian curvature equal to $\frac{4\pi}{|\Sigma|}$. 

Now, assuming $\text{Ric}_g\geq 2$, we can apply  \cite{Lawson70} (in fact we only need positive Ricci curvature) and conclude that $\Sigma$ divides $M$ into two components, $\Omega_1$ and $\Omega_2$, such that $\partial \Omega_1  = \partial \Omega_2= \Sigma$. 
At this point we are in position to apply 
\cite[Theorem 3]{hangwang} to conclude that if $\Omega_1 $ is the region determied by the mean convex side of $\Sigma$, then it is isometric to a ball  in $ \mathbb{S}^3_+$ and this concludes the proof.
\qed
\begin{corollary}
    Let $(M^3, g)$ be a complete Riemannian manifold with $S_g\geq 6$, and let $\Sigma$ be a stable minimal sphere  in $M^3$ for volume preserving variatoins. If $m_H(\Sigma) = 0$ and $\Sigma$ is approximately a round sphere, then  $\Sigma$ is isometric to $\mathbb{S}^2$. Moreover, if we assume $\textrm{Ric}_g \geq 2$, then  $M^3$ is isometric to $\mathbb{S}^3.$
\end{corollary}
\begin{proof}
    It follows from the proof of the above theorem that if $\Sigma$ is minimal, then $\Sigma$ is totally geodesic and is isometric to $\mathbb{S}^2$.
Now, assuming $\text{Ric}_g\geq 2$, we can apply  \cite{Lawson70} and conclude that $\Sigma$ divides $M$ into two components, $\Omega_1$ and $\Omega_2$, such that $\partial \Omega_1  = \partial \Omega_2= \Sigma$. 
At this point we are in position to apply the classical theorem of Hang-Wang 
\cite[Theorem 2]{hangwang} to conclude that both $\Omega_1$ and $\Omega_2$ are isometric to $\mathbb{S}^3_+$, and, therefore, $M^3$ is isometric to  $\mathbb{S}^3$.
\end{proof}
\subsection{Proof of Theorem \ref{teo2}}
Let $\varphi : \Sigma \rightarrow \mathbb{S}^2 \subset \mathbb{R}^3$ be the conformal map introduced in Proposition \ref{prop1sun}, with the condition that $\int_\Sigma \varphi_i , da = 0$.

Based on our initial hypothesis, there exists an isometry 
$\rho: \Sigma \rightarrow \Sigma$, satisfying $\rho^2 = \text{id}$ and $\rho(x) \neq x$ 
for $x \in \Sigma$.
Consequently, we can define $\tilde{\rho} = \varphi \circ \rho \circ \varphi^{-1}: \mathbb{S}^2 \rightarrow \mathbb{S}^2$, which is a conformal map.

This leads us to the following  equalities:
\begin{eqnarray*}
     \varphi^*(e^u g_0) &=& g = \rho^*g = \rho^*\varphi^*(e^u g_0) = (\varphi \circ \rho)^*(e^u g_0) \\
     &=& (\tilde{\rho} \circ \varphi)^*(e^u g_0) = \varphi^*\tilde{\rho}^*(e^u g_0).
\end{eqnarray*}

Therefore, we have $e^u g_0= \tilde{\rho}^*(e^u g_0)$, which implies that $\tilde \rho$ is an isometry with respect to the metric $e^u g_0$.

Now, consider $z\in C^\infty(\mathbb S^2)$ such that $\tilde{\rho}^*g_0 = e^zg_0$. Then, we have the relationship $u = u\circ \tilde{\rho} + z$. Since $(\mathbb{S}^2,\tilde{\rho}^*g_0)$ has constant Gaussian curvature equal to 1, and as $\tilde{\rho}: (\mathbb{S}^2,\tilde{\rho}^*g_0) \rightarrow (\mathbb{S}^2,g_0)$ is an isometry, it follows that $z$ satisfies $\Delta z = 2(1-e^z)$.

Then, we can proceed as follows
\begin{eqnarray*}
    6 \bigg( 1 -\frac{24\pi}{|\Sigma|}e^u\bigg)  &=& \Delta u  = \Delta ( u \circ \tilde{\rho} ) + \Delta z = ((\Delta u ) \circ \tilde{\rho})e^z + \Delta z \\
    &=& 6 (e^z - \frac{24\pi}{|\Sigma|}e^{u\circ \tilde{\rho }+z})+ \Delta z =6 (e^z - \frac{24\pi}{|\Sigma|} e^u) + 2 (1-e^z) \\
    &=& 6(1- \frac{24\pi}{|\Sigma|} e^u) - 4(1-e^z).
\end{eqnarray*}

Thus, $e^z = 1$, $\tilde{\rho}^*g_0 = g_0$, and $\tilde{\rho}$ is an isometry of $\mathbb{S}^2$ with respect to the metric
$g_0$. 
Consequently, there exists an orthogonal matrix $A \in \mathcal{O}(3)$ such that $\tilde{\rho}x = Ax$. Since $\rho^2 = \text{id}$, we have that $\tilde{\rho}^2 = \text{id}$, and $A^2 = I_3$. The eigenvalues of $A$ must be $\pm 1$. In fact, if $Ax = \lambda x$, $x \ne 0$, we have $x = A^2x = \lambda^2 x$. Since $\rho(x) \ne x$ for $x \in \Sigma$, we have $\tilde{\rho}(x) \ne x$ for $x \in \mathbb{S}^2$, and $1$ cannot be an eigenvalue of $A$. Therefore, the eigenvalue of $A$ must be $-1$, so $A$ is similar to a diagonal matrix $-I_3 = Q^{-1}AQ$, where $Q$ is an invertible matrix, which implies that $A = -I_3$ and $\tilde{\rho}(x) = -x$. Hence, $u(x) = u\circ \tilde{\rho}(x) + z= u(-x)$.

By making again the change of variable $u = w + v$, where $e^v = \frac{|\Sigma|}{4\pi}$, we find that $w$ satisfies (\ref{3.17sun}), and $w(x) = w(-x)$. Using \cite[Theorem 1]{shi2019}, we can conclude that $w = 0$. Therefore, $\Sigma$ is isometric to $\mathbb{S}^2(\sqrt{\frac{|\Sigma|}{4\pi}})$.

Assuming that $\text{Ric}_g\geq 2$, we can proceed analogously to the proof of Theorem \ref{teo1} to conclude that $M^3$contains a region isometric to a compact set ${\Omega} \subset \mathbb{S}^3_+$ such that $\Sigma$ is isometric to $\partial \Omega$.\qed

\begin{remark}\label{final}
Using \eqref{3.18} and the result in \cite[Lemma 10]{shi2019}, we can derive $K_\Sigma \frac{|\Sigma|}{4\pi} \leq 2$. Consequently, the assumption of even symmetry can be replaced by this Gaussian curvature estimate. This results in an explicit bound in Theorem \ref{teo1}.
\end{remark}

\begin{corollary}
  Let $(M^3, g)$ be a complete Riemannian manifold with $S_g\geq 6$, and let $\Sigma$ be a  stable minimal sphere  in $M^3$ for volume preserving variatoins. If $m_H(\Sigma) = 0$ and  $\Sigma$ has the even symmetry, then  $\Sigma$ is isometric to $\mathbb{S}^2$. Moreover, if we assume $\textrm{Ric}_g \geq 2$, then  $M^3$ is isometric to $\mathbb{S}^3.$  
\end{corollary}
\bibliographystyle{amsplain}
\bibliography{bibliography}

\end{document}